\documentclass[a4paper,12pt]{article}
\usepackage[utf8]{inputenc} 
\usepackage{amsmath}
\usepackage{amsthm}
\usepackage{amssymb}
\usepackage{mathrsfs}
\usepackage{graphicx}
\usepackage{color}
\usepackage[textsize=small,textwidth=3cm]{todonotes}

\usepackage[icelandic,english]{babel}
\usepackage{t1enc}
\usepackage{titlesec}
\usepackage[colorlinks=false,
pdfauthor={Benedikt Steinar Magnusson},
pdftitle={Analytic Discs, Global Extremal Functions and Projective Hulls in Projective Space
}]{hyperref}

\renewcommand{\P}{{\mathbb  P}}
\newcommand{\C}{{\mathbb  C}}
\newcommand{\R}{{\mathbb  R}}

\newcommand{\D}{{\mathbb D}}

\newcommand{\T}{{\mathbb  T}}
\newcommand{\A}{{\mathcal A}}

\newcommand{\F}{{\mathcal F}}
\renewcommand{\L}{{\mathcal L}}
\renewcommand{\O}{{\mathcal O}}
\newcommand{\PSH}{{\operatorname{\mathcal {PSH}}}}

\renewcommand{\phi}{\varphi}
\renewcommand{\epsilon}{\varepsilon}

 \titleformat{\section}{\normalfont\filcenter}{\thesection}{1em}{\uppercase}

\newtheorem{theorem+}           {Theorem}      [section]
\newtheorem{definition+}  [theorem+]  {Definition}
\newtheorem{lemma+}  [theorem+]  {Lemma}
\newtheorem{corollary+}  [theorem+]  {Corollary}
\newtheorem{proposition+}  [theorem+]  {Proposition}
\newtheorem{example+}  [theorem+]  {Example}
\newtheorem{question+}  [theorem+]  {Question}

\newenvironment{theorem}{\begin{theorem+}\sl}{\end{theorem+}\rm}
\newenvironment{definition}{\begin{definition+}\rm}{\end{definition+}\rm}
\newenvironment{lemma}{\begin{lemma+}\sl}{\end{lemma+}\rm}
\newenvironment{corollary}{\begin{corollary+}\sl}{\end{corollary+}\rm}
\newenvironment{proposition}{\begin{proposition+}\sl}{\end{proposition+}\rm}

\title{Analytic Discs, Global Extremal Functions and Projective Hulls in Projective Space
}
\author{Benedikt Steinar Magn\'usson }
\date{\today}

\begin{document}
\maketitle
\begin{abstract}
 Using a recent result of Lárusson and Poletsky regarding 
 plurisubharmonic subextensions we prove a disc 
 formula for the quasiplurisubharmonic global extremal function for domains 
 in $\P^n$.
 As a corollary 
 we get a characterization of the projective hull for connected compact sets in 
 $\P^n$ by 
 the existence of analytic discs.
 %Finally we give a new proof of
 %the disc formula for the Siciak-Zahariuta extremal functions for domains. 
\end{abstract}

%%%%%%%%%%%%%%%%%%%%%%%%%%%%%%%%%%%%%%%%%%%%%%%%%%%%%%%%%%%%%%%%%%%%%%%%%%%%%%
%%%%%%%%%%%%%%%%%%%%%%%%%%%%%%%%%%%%%%%%%%%%%%%%%%%%%%%%%%%%%%%%%%%%%%%%%%%%%%
\section{Introduction}\label{sec:intro}
The global extremal function, also called the
Siciak-Zahariuta extremal function,
has proven very useful for pluripotential theory in $\C^n$, see
\cite[\S 13]{Kis:psh_scv} and \cite[\S 5]{Klimek} for an overview
of the applications.
We are however most interested in its counterpart in the theory of quasiplurisubharmonic 
functions on compact manifolds. The quasiplurisubharmonic global extremal function was 
defined by Guedj and Zeriahi \cite{GueZer:2005} and has already proven useful,
most notably in connection with projective hulls \cite{HarLaw:2006}. 
But the projective hull of a compact set in $\P^n$ 
is the natural generalization of 
the polynomial hull in $\C^n$.

We start
by looking at a recent result of Lárusson and Poletsky
\cite{LarPol:2012} regarding plurisubharmonic subextensions 
for domains in $\C^n$. There we make a small observation regarding their results
(Corollary \ref{cor:main}), and we also define a disc structure for 
sets in $\C^{n+1} \setminus \{0\}$ with some nice properties.
This is done in Section \ref{sec:subextensions}.

In Section \ref{sec:omega} we turn our attention
to quasiplurisubharmonic function, or $\omega$-plurisubharmonic functions,
on $\P^n$ which we denote by $\PSH(\P^n,\omega)$. The current $\omega$ is here
the Fubini-Study K\"ahler form. 
Using the results from the Section \ref{sec:subextensions}
we prove a disc formula for the
global extremal function for a domain $W \subset \P^n$ 
(Theorem \ref{th:omega}),
\begin{multline*}
  \sup\{ u(x) ; u \in \PSH(\P^n,\omega), u|_W \leq \phi \} = \\
  \inf \left\{ -\frac{1}{2\pi}\int_\D \log|\cdot|\, f^*\omega +
  \int_\T \phi \circ f\, d\sigma\ ;\ f \in \A_{\P^n}^W, f(0) = x \right\}.
\end{multline*}

In Section \ref{sec:projective_hulls} we show an applications 
of these results (Theorem \ref{th:hulls}). There we show that the points 
in the projective hull $\hat K$
of connected compact sets $K \subset \P^n$
can be characterized by the existence of analytic discs 
with specific properties. That is, for $\Lambda >0$ and a connected compact subset 
 $K\subset \P^n$ the following is equivalent
 for a point $x \in \P^n$:
 \begin{itemize}
  \item[(A)] $x \in \hat K(\Lambda)$
  \item[(B)] For every $\epsilon >0$ and every neighbourhood $U$ of $K$ 
  there exists a disc
  $f \in \A_{\P^n}^U$ such that $f(0)=x$ and 
  $$
		- \frac{1}{2\pi}\int_\D \log|\cdot| f^*\omega < \Lambda + \epsilon.
	$$
 \end{itemize}
 Here, $\hat K(\Lambda)$, $\Lambda > 0$, are specific subsets of $\hat K$ with $\hat K
 = \cup_\Lambda \hat K(\Lambda)$ which are defined using the best constant function for $K$ (see
 \cite[\S 4]{HarLaw:2006}).

Finally, in Section \ref{sec:S-Z}, we see how the methods presented in 
Section \ref{sec:subextensions} and \ref{sec:omega} 
work also for other currents, in particular for the current of integration
for the hyperplane at infinity $H_\infty \subset \P^n$. 
This gives rise to a disc formula for 
the Siciak-Zahariuta extremal function for a domain $W\subset \C^n$, 
\begin{multline}\label{eq:S-Z}
  \sup\{ u(x) ; u \in \L, u|_W \leq \phi \} = \\
  \inf \left\{ -\sum_{a \in f^{-1}(H_\infty) } \log|a| +
  \int_\T \phi \circ f\, d\sigma\ ;\ f \in \A_{\P^n}^W, f(0) = x \right\},
\end{multline}
where $\L$ is the Lelong-class of plurisubharmonic functions of 
logarithmic growth. This formula was first proved in \cite{LarSig:2005}, 
when $\phi=0$, and in \cite{MagSig:2007}, when $\phi$ is upper semicontinuous.
For the case when $W$ is not connected see \cite{LarSig:2009}.

\bigskip

We now must establish some notation. 
We assume $X$ is a complex manifold, 
here the role of $X$ will either be played 
by subsets of affine space or projective space.
Let $\A_{X}$ denote the family of 
closed analytic discs in $X$, that is 
continuous maps $f:\overline \D \to X$ which are 
holomorphic on the unit disc $\D$. 
Assume $W \subset X$, then 
$\A_{X}^W$ is the subset of discs in $\A_X$ which map the unit circle 
$\T$ into $W$.

If $H$ is a \emph{disc functional}, that is a function from a subset 
of $\A_X$ to $[-\infty,+\infty]$, then its 
\emph{envelope} with respect to the family 
$\mathcal C \subset \A_X$ is defined as the function
$$
	E_{\mathcal C}H(x) = \inf\{ H(f) ; f\in \mathcal C, f(0) = x \}.
$$
The domain of $E_{\mathcal C}H$ is all the points $x\in X$
such that $\{ f\in \mathcal C ; f(0)=x\}$ is non-empty.

For convenience we write $E$ for $E_{\A_X}$  and $E_W$ for $E_{\A_X^W}$.

The standard example of a disc functional is the \emph{Poisson disc 
functional} $H_\phi:\A_X \to [-\infty,+\infty]$ for a 
function $\phi:X \to \R\cup \{-\infty\}$. It 
is defined as  $H_\phi(f) = \int_\T \phi\circ f\, d\sigma$.
The measure $\sigma$ is the arclength measure on $\T$ normalized to one.
Other disc functional we will use are the 
Poisson disc functional for the class of $\omega$-plurisubharmonic
functions
$$
	H_{\omega,\phi}(f) = 
	- \frac{1}{2\pi}\int_\D \log|\cdot|\, f^*\omega 
	+ \int_\T \phi\circ f\, d\sigma ,
$$
where $f^*\omega$ is the pullback of $\omega$ by $f$. 
In our case the $(1,1)$-current $\omega$ will be 
the Fubini-Study K\"ahler form on $\P^n$. However, in Section \ref{sec:S-Z}
we look briefly at other currents on $\P^n$, specially the case 
when the current is the current
of integration for the hyperplane at infinity. 
% $H_\infty \subset \P^n$.
% In the later case the term $\int_\D \log|\cdot|\, f^*\omega$
% becomes $\sum_{a\in f^{-1}(H_\infty)} m_a\log|a|$, where $m_a$ is 
% the multiplicity of the intersection of with $H_\infty$, i.e.~the 
% order of the zero of the zero component $f_0$ of $f=[f_0:\cdots:f_n]$.

If $\phi$ is a function defined on $W \subset X$ then we let 
$$
	\F_\phi = \{ u \in \PSH(X) ; u|_W \leq \phi \},
$$
$$
	\F_{\omega,\phi} = \{ u \in \PSH(X,\omega) ; u|_W \leq \phi \}.
$$

\medskip

\textbf{Remark:} 
Although it is more traditional to look at analytic discs which are holomorphic 
in a neighbourhood of the closed unit disc we are only assuming the discs are
continuous to the boundary. This does in fact not alter the results 
obtained here since every disc holomorphic in a neighbourhood of 
$\overline \D$ is clearly in $\A_X$, and conversely if $f\in \A_X$
then $f(r\cdot)$, $r<1$ is a family of discs in holomorphic in a 
neighbourhood of $\D$ such that
$H_{\omega,\phi}(f(r\cdot)) \to H_{\omega,\phi}(f)$, when $r\to 1^-$. 
The reason for this is that the authors of \cite{LarPol:2012} applied a 
results of Forstneri\v{c} 
\cite{For:2007} which uses discs which are
only continuous up to the boundary.

%%%%%%%%%%%%%%%%%%%%%%%%%%%%%%%%%%%%%%%%%%%%%%%%%%%%%%%%%%%%%%%%%%%%%%%%%%%%%%
%%%%%%%%%%%%%%%%%%%%%%%%%%%%%%%%%%%%%%%%%%%%%%%%%%%%%%%%%%%%%%%%%%%%%%%%%%%%%%
\section{Plurisubharmonic subextensions}
\label{sec:subextensions}
We now turn our attentions to the work done by Lárusson and Poletsky
\cite{LarPol:2012}.

Their setting is the following. For domains $W\subset X\subset \C^n$
and an upper semicontinuous function $\phi:W \to \R \subset \{-\infty\}$, 
they consider the function 
$$
	\sup \F_\phi(x) = \sup \{ u(x) ; u \in \F_\phi \},
$$
which is largest plurisubharmonic function on $X$, dominated by 
$\phi$ on $W$.
Under sufficient condition on $W$ and $X$ they prove a disc 
formula for this function, namely that $\sup F_\phi = E_W H_\phi$, 
or if we write it out
$$
	\sup\{ u(x) ; u \in \PSH(X), u(\T) \subset W \} 
	= \inf \left\{ \int_\T \phi \circ f\, d\sigma ; f \in \A_X^W, f(0) = x \right\}.
$$

Before we look at this formula in more detail we need the following 
definitions.

\begin{definition}\label{def:centre-homo}
We say that two discs $f_0$ and $f_1$ in $\A_X^W$ with $f_0(0) = f_1(0)$ are 
\emph{centre-homotopic} if there is a continuous map $f:\overline \D \times
[0,1] \to X$ such that
\begin{itemize}
    \item $f(\cdot,t) \in \A_X^W$ for all $t\in [0,1]$, 
    \item $f(\cdot,0) = f_0$ and $f(\cdot,1) = f_1$, 
    \item $f(0,t) = f_0(0)  =f_1(0)$ for all $t\in [0,1]$.
\end{itemize}
\end{definition}

\begin{definition}
If $ W \subset X$, then a \emph{$W$-disc structure} on $X$ is a family $\beta = (\beta_\nu)_\nu$ of continuous
maps $\beta_\nu : U_\nu \to \A_X^W$, where $(U_\nu)_\nu$ is an open covering of $X$, such that
\begin{itemize}
  \item $\beta_\nu(x)(0) = x$ for all $x\in U_\nu$
  (i.e.~$x$ is mapped to a disc centred at $x$),
  %(i.e.~each point is mapped to a disc centred at that point)
  \item If $x \in U_\nu \cap U_\mu$ then $\beta_\nu(x)$ and $\beta_\mu(x)$ are centre-homotopic.
%\flushright (S)
\end{itemize}
Furthermore, if there is $\mu$ such that $U_\mu = W$ and $\beta_\mu(w)(\cdot) = w$ for 
every $w \in W$ (i.e.~$\beta_\mu(w)$ is the constant disc), then we say that the disc structure is \emph{schlicht}.
\end{definition}

For a $W$-disc structure $\beta$ we let $\mathcal B \subset \A_X^W$ 
denote the family of discs $ \mathcal B = \cup_\nu \beta_\nu(U_\nu)$.

\begin{lemma}[\cite{LarPol:2012}, Lemma 2]\label{lem:Lemma2}
Let $W \subset X$ be domains in $\C^n$, and $\beta$ a $W$-disc structure on $X$. 
If $\phi : W \to \R \subset \{-\infty\}$ is an upper semicontinuous function then
$$
  E_W H_\phi \leq EH_{E_{\mathcal B} H_\phi}.
$$
\end{lemma}

If $\beta$ is a $W$-disc structure on $X$ and $\phi$ is upper semicontinuous 
then it follows easily from the continuity of the $\beta_\nu$'s that 
$E_{\mathcal B}H_\phi$ is an upper semicontinuous function on $X$. 

\begin{theorem}[\cite{LarPol:2012}, Theorem 3]
Let $W \subset X$ be domains in $\C^n$ and assume $\beta$ is a schlicht $W$-disc
structure on $W$, then
$$
  \sup \F_\phi = E_W H_\phi.
$$\label{th:L-P}
\end{theorem}
\begin{proof} The formula follows from the following inequalities,
$$
  \sup \F_\phi 
  \leq E_W H_\phi 
  \leq EH_{E_{\mathcal B} H_\phi} 
  \leq \sup \F_\phi.
$$
The first inequality follows from the subaverage property of the 
subharmonic function $u\circ f$. 
If $u \in \F_\phi$ and $f \in \A_X^W$, $f(0)=x$ then 
$$
	u(x) = (u\circ f)(0) \leq \int_\T u\circ f\, d\sigma \leq \int_\T
\phi \circ f\, d\sigma = H_\phi(f).
$$ 
Taking supremum on the left hand
side over $u \in \F_\phi$ and infimum on the right hand side over $f\in \A_X^W$
gives the inequality.

Lemma \ref{lem:Lemma2} gives the second inequality.

The last inequality follows from the fact that the  function 
$EH_{E_{\mathcal B} H_\phi}$
is plurisubharmonic by Poletsky's theorem 
\cite{Pol:1991,Pol:1993} and not greater than $\phi$ since 
$\beta$ is schlicht. It is therefore in the class $\F_\phi$ we take
supremum over.
\end{proof}

Note that $EH_{E_{\mathcal B} H_\phi}$ is always plurisubharmonic when $\phi$ is upper
semicontinuous because then ${E_{\mathcal B} H_\phi}$ is upper semicontinuous.
The only requirements for the last inequality are therefore that
$EH_{E_{\mathcal B} H_\phi} \leq \phi$ on $W$. From this small observation 
we get the following corollary.

\begin{corollary}\label{cor:main}
Let $W \subset X$ be domains in $\C^n$ and assume $\beta$ is a $W$-disc
structure on $W$ such that $E_{\mathcal B} H_\phi \leq \phi$ on $W$, then
$$
  \sup \F_\phi = E_W H_\phi.
$$
\end{corollary}

\begin{proof}
A fundamental property of the envelops of the Poisson disc functional 
is that 
$EH_\psi \leq \psi$ and therefore, with $\psi = E_{\mathcal B} H_\phi$,
$$
  EH_{E_{\mathcal B} H_\phi} \leq E_{\mathcal B} H_\phi \leq \phi.
$$
The rest of the proof is then same the as in the proof of Theorem \ref{th:L-P}.
\end{proof}

We will now construct a disc structure on a certain class of 
sets in $X =\C^m \setminus \{0\}$ satisfying the condition in 
Corollary \ref{cor:main}. In Section \ref{sec:omega} and \ref{sec:projective_hulls}
we let 
$m = n+1$ and look at $\P^n$ using homogeneous coordinates 
$\pi: \C^{n+1}\setminus \{0\} \to \P^n$.

\begin{definition}
 A set $W \subset \C^m$ is a \emph{complex cone} if $\lambda x \in W$ for
 every $\lambda \in \C$ and $x \in W$. 
\end{definition}
Later, when we talk about a complex cone in $\C^{m} \setminus \{0\}$ it
is simply a complex cone in $\C^{m}$ with $0$ removed.

\begin{definition}
 Assume $W$ is a complex cone. 
 A function $\phi:W \to \R \cup \{-\infty\}$ is called \emph{logarithmically
 homogeneous} if 
 $$
	\phi(\lambda x) = \phi(x) + \log|\lambda|,
 $$
 for every $\lambda \in \C$ and $x \in W$.
\end{definition}
Note that every function on a complex cone in $\C^{m}\setminus \{0\}$ 
which is logarithmically homogeneous extends automatically over 0 and takes
the value $-\infty$ there.

\begin{lemma}\label{lem:n+1_disc_structure}
 Assume $W\subset \C^m \setminus \{0\}$, $m\geq 2$ is a complex cone and a 
 domain, and assume that $\phi:W \to \R \cup \{-\infty\}$
 is an upper semicontinuous function which is logarithmically homogeneous.
 Then there exists a $W$-disc structure
 $\mathcal B$ in $\C^m \setminus \{0\}$ such that 
 $E_{\mathcal B} H_\phi \leq \phi$.
\end{lemma}
\begin{proof}
For each $w \in W$ let $U_w  = \C^m \setminus \{\lambda w ; \lambda \in \C\}$  
and define the analytic discs $\beta_w(x)$ by
\begin{align*}
  f_{x,w}(t) = \beta_w(x)(t) 
  &= \left(\frac{\|x-w\|}r - \frac{r}{\|x-w\|}\right)t w
    + \left(1+ \frac{r}{\|x-w\|}t\right) x
\end{align*}
where 
$$
  r = \min\left\{ \frac{\|x-w\|}{1+\|x-w\|}, \frac{d(w, W^c)}2 \right\},
$$ 
and $W^c$ is the complement of $W$ in $\C^m \setminus \{0\}$.

It is more convenient to write the formula for these discs in the following way
$$
  f_{x,w}(t) 
  = \left\{ \begin{array}{ll}
    \left(1 + \frac{\|x-w\|}r t\right) 
    \underbrace{
    \left[ w + \left( \frac{\|x-w\| + rt}{r + \|x-w\|t}\right)
    \frac r{\|x-w\|} (x-w) \right]
    }_{(\star)}
    & \text{if } t \neq -\frac r{\|x-w\|} \\
    %\in \D_{asdf} \setminus \{ -\frac r{\|x-w\|} \}\\
    \left(1- \frac{r^2}{\|x-w\|^2}\right)(x-w) 
    & \text{if } t = -\frac r{\|x-w\|}
    \end{array}\right.
$$
Then we see that $0$ is mapped to $x$. Furthermore, the factor in the 
brackets $(\star)$ maps the closed unit disc into the 
complex line through $x$ and $w$, and maps the unit 
circle into a circle with centre
$w$ and radius $r$. This can be seen from the fact that if $t\in \T$ then
$| \frac{\|x-w\| + rt}{r + \|x-w\|t} |=1$ and 
$$
    \| w - (\star)\| = r < d(w,W^c).
$$
This implies that for $t\in \T$ the value 
$f_{x,w}(t) = (1+\frac{\|x-w\|}r t)(\star)$ 
is also in $W$ since $W$ is a complex cone. 

Note also that $0$ is not in the image of $f_{x,w}$ because by the definition 
of $U_w$ the complex line through $x$ and $w$ does not include $0$.

To show that this is a $W$-disc structure we need to show that every two discs
with the same centre are centre-homotopic, that is $f_{x,w}$ and $f_{x,w'}$ are 
centre-homotopic for every $w, w' \in W$. Since $W$ is connected
the set $W \setminus \{ \lambda x ; \lambda \in \C \}$ is also connected and 
path connected. Therefore 
there is a path $\gamma : [0,1] \to W 
\setminus \{ \lambda x ; \lambda \in \C \}$ such that 
$\gamma(0) =w$ and $\gamma(1) = w'$.
Define the map $f: \overline \D \times [0,1] \to \C^{n+1} \setminus \{0\}$ by
$$
  f(t,s) = f_{x,\gamma(s)}.
$$
The function $f$ clearly satisfies all the conditions in Definition 
\ref{def:centre-homo}, which means
that we have defined a $W$-disc structure 
$\mathcal B = \cup_{w\in W} \{ f_{x,w} ; x \in U_w \}$.

We now show that $E_{\mathcal B} H_\phi \leq \phi$
on $W$. Fix $x \in W$ and $\epsilon > 0$. Since $\phi$ is 
upper semicontinuous there is an open neighbourhood $U$ of $x$ such that
$\phi|_U \leq \phi(x) + \epsilon/2$. Then select $w$ close enough to 
$x$ so that
\begin{itemize}
 \item $\frac 1{2\pi}\log(1 +\|x-w\|) < \epsilon/2$,
 \item $r = \min\left\{ \frac{\|x-w\|}{1+\|x-w\|}, \frac{d(w, W^c)}2 \right\}$
    is equal to $\frac{\|x-w\|}{1+\|x-w\|}$,
 \item the disc on the complex line through $x$ and $w$ with
    centre $w$ and radius $r$ (defined as above) is in $U$.
\end{itemize}
Then by using the properties of $\phi$, the properties of the term $(\star)$, 
and the Riesz representation formula and we see that
\begin{align*}
  E_{\mathcal B} H_\phi(x) 
  &\leq \, \int_\T \phi \circ f_{x,w}\, d\sigma  \\
  &\leq \, \int_\T \phi \left( \left(1 + \frac{\|x-w\|}r t\right)(\star)\right)\, d\sigma  \\
  &\leq \, \int_\T \phi ((\star))\, d\sigma  + \int_\T \log\left|1 + \frac{\|x-w\|}r t\right|\, d\sigma\\
  &\leq \,  \sup_{U} \, \phi  - \frac 1{2\pi}\log\left|-\frac{r}{\|x-w\|}\right| \\
  &\leq \, \phi(x) + \frac \epsilon 2 + \log( 1+\|x-w\|) \leq \, \phi(x) + \epsilon.
\end{align*}
This holds for every $\epsilon > 0$, hence $E_{\mathcal B} H_\phi \leq \phi$.

\end{proof}
It should be noted here that this disc structure defined here 
is under heavy influence from
the set of ``good discs'' used in both \cite{LarSig:2005} and
\cite{MagSig:2007} for the original proof of (\ref{eq:S-Z}).

% For completeness' sake we also include the following theorem giving another
% sufficient condition for which $\sup \F_\phi = E_W H_\phi$ holds.
% It is however not needed for our work here. 
% 
% \begin{theorem}[\cite[Theorem 4]{LarPol:2012}
% Let $W \subset X$ be domains in $\C^n$ and assume $\A_X^W$ has a connected component 
% $\mathcal B$ such that
% \begin{itemize}
%   \item $\beta$ covers $X$, and
%   \item If two discs in $\beta$ have the same centre then they are centre-homotopic.
% \end{itemize}
% Then, for every upper semicontinuous function $\phi: W \to \R \cup \{-\infty\}$,
% $$
%   \sup \F_\phi = E_W H_\phi.
% $$
% \end{theorem}

%%%%%%%%%%%%%%%%%%%%%%%%%%%%%%%%%%%%%%%%%%%%%%%%%%%%%%%%%%%%%%%%%%%%%%%%%%%%%
%%%%%%%%%%%%%%%%%%%%%%%%%%%%%%%%%%%%%%%%%%%%%%%%%%%%%%%%%%%%%%%%%%%%%%%%%%%%%%
\section{Disc formula for the global relative extremal function in
 $(\P^{\lowercase{n}},\omega)$}\label{sec:omega}

We let $\omega$ be the Fubini-Study K\"ahler form 
for $\P^n$. Recall that an upper semicontinuous function $u$ on $\P^n$
is called $\omega$-plurisubharmonic (or quasiplurisubharmonic) if
$dd^c u + \omega \geq 0$. We denote the family of $\omega$-plurisubharmonic
function on $\P^n$ by $\PSH(\P^n,\omega)$. 

If $f \in \A_{\P^n}$ then there is a well defined pullback
of $\omega$ by $f$, denoted $f^*\omega$. It is defined locally by
$\Delta \psi\circ f$ where $\psi$ is a local potential of $\omega$,
i.e.~$\psi$ is a plurisubharmonic function such that $dd^c \psi = \omega$.
For a more details
about $\omega$-plurisubharmonic functions and analytic discs see
\cite[\S 2]{Mag:2012}.

The pullback of the current $\omega$ to $\C^{n+1} \setminus 0$ satisfies
$$
	\pi^*\omega = dd^c \log\| \cdot \|,
$$
where $\pi: \C^{n+1}\setminus \{0\} \to \P^n$ is the projection
$$
    \pi(z_0,z_1,\ldots,z_n) = [z_0:z_1:\cdots:z_n].
$$

This implies that if $\tilde f \in \A_{\C^{n+1}\setminus \{0\}}$ and we let
$f = \pi \circ \tilde f \in \A_{\P^n}$ then
$$ 
	f^*\omega = \Delta \log\|\tilde f\|.
$$

\begin{proposition}\label{prop:P2Cn+1}
 There is a one to one correspondence between the class $\PSH(\P^n,\omega)$
 and $\{ u\in \PSH(\C^{n+1}\setminus\{0\}) ; 
 u \text{ logarithmically homogeneous} \}$.
\end{proposition}
\begin{proof}
 If $v\in \PSH(\P^n,\omega)$ then $u = v \circ \pi + \log\|\cdot\|$ is 
 pluri\-sub\-harmonic on $\C^{n+1}\setminus\{0\}$ since
 $$
	dd^c(v \circ \pi + \log\|\cdot\|) = \pi^*(dd^c v + \omega) \geq 0.
 $$
 Conversely, if $u$ is in the later class, then for 
 $z\in \C^{n+1}\setminus\{0\}$, 
 $v([z]) = u(z) - \log\|z\|$ is a well defined function on $\P^n$ since
 $$
	u(\lambda z) - \log\|\lambda z\| = u(z) + \log|\lambda| - \log\|\lambda z\|
 = u(z) - \log\|z\|, \quad \lambda \in \C^*.
 $$
 Furthermore, $v$ is $\omega$-plurisubharmonic
 by the same calculations as above.
\end{proof}

\begin{lemma}\label{eq:supFloghom}
 Assume $X$ is a domain and a complex cone in $\C^{n+1}\setminus \{0\}$.
 If $\phi :X \to \R \cup \{-\infty\}$ is upper semicontinuous and
 logarithmically homogeneous then the function
 $\sup \F_{\phi}$,
 $$
	\sup \F_{\phi}(x) = \sup\{ u(x) ; u \in \PSH(\C^{n+1}\setminus \{0\}, 
	u|_X \leq \phi \}
 $$
	is also logarithmically homogeneous.
\end{lemma}
\begin{proof}
 Since  function $\sup \F_\phi$ is dominated by $\phi$, then
 $$
	\sup \F_\phi(\lambda x) + \log|\lambda| \leq \phi(\lambda x) + \log|\lambda| 
	= \phi(x)
  $$
  and furthermore, since $\sup \F_\phi$ is plurisubharmonic, 
  the function $x \mapsto \sup \F_\phi(\lambda x)$ is also plurisubharmonic
  and therefore in $\F_\phi$.
  This implies, by the definition of $\sup \F_\phi$, that 
  $$
		\sup \F_\phi(\lambda x) + \log|\lambda| \leq \sup \F_\phi(x), \qquad 
		\lambda \in \C^*.
	$$
	By setting $\lambda^{-1}$ instead of $\lambda$ and $\lambda x$ instead of
	$x$ we get the inverted inequality, hence 
	$\sup \F_\phi(\lambda x) + \log|\lambda| = \sup \F_\phi(x)$.
\end{proof}

Now we prove the main result.

\begin{theorem}\label{th:omega}
Let $\omega$ by the Fubini-Study K\"ahler form. If $W\subset \P^n$ is a domain and $\phi:W \to \R \cup \{-\infty\}$ 
is an upper semicontinuous function then 
\begin{multline}\label{eq:omega}
  \sup\{ u(x) ; u \in \PSH(\P^n,\omega), u|_W \leq \phi \} = \\
  \inf \left\{ -\frac{1}{2\pi}\int_\D \log|\cdot|\, f^*\omega +
  \int_\T \phi \circ f\, d\sigma\ ;\ f \in \A_{\P^n}^W, f(0) = x \right\}.
\end{multline}
or, using the notation from Section \ref{sec:intro},
$\sup \F_{\omega,\phi} = E_W H_{\omega,\phi}$.
\end{theorem}

\begin{proof}
Define the complex cone
$\tilde W = \pi^{-1}(W)$ and define the logarithmically homogeneous function
$\tilde \phi: \tilde W \to \R \cup \{-\infty\}$ by
$$
  \tilde \phi(z)  
  = \phi ([z_0:z_1:\cdots:z_n]) + \log \|z\|.
$$

Fix $\tilde f \in \A_{\C^{n+1} \setminus \{0\}}^{\tilde W}$ and let 
$f = \pi \circ \tilde f$, $f \in \A_{\P^n}^W$.
Let $z = \tilde f(0)$ which implies
$\pi(z) = f(0)$. 
Then 
$$
  \tilde \phi \circ \tilde f 
  = \phi \circ f  + \log\|f\|
$$
By the Riesz representation formula for the function $\log\|\tilde f\|$
at the point $0$,
\begin{align*}
  \log\|\tilde f(0)\| 
  &= \frac{1}{2\pi}\int_\D \log|\cdot|\, \Delta \log\|\tilde f\| 
		+ \int_\T \log\|\tilde f\|\, d\sigma \\
  &= \frac{1}{2\pi}\int_\D \log|\cdot|\, \tilde f^*(\pi^*\omega) 
		+ \int_\T \log\|\tilde f\|\, d\sigma
\end{align*}
Since $\tilde f^*(\pi^*\omega) = (\pi\circ \tilde f)^*\omega = f^*\omega$, 
this shows that
 $$
  - \int_\T \log\|\tilde f\|\, d\sigma  -\frac{1}{2\pi}\int_\D \log|\cdot|\, f^*\omega 
  = -\log\|\tilde f(0)\| = - \log\|z\|.
 $$
Using the three previous equalities we derive that 
\begin{align*}
	H_{\omega,\phi}(f) 
	&= \int_\T \phi \circ f\, d\sigma - \frac{1}{2\pi}\int_\D \log|\cdot|\, f^*\omega\\
	&= \int_\T \tilde \phi \circ \tilde f\, d\sigma - \int_\T \log\|\tilde f\|\, d\sigma - \frac{1}{2\pi}\int_\D \log|\cdot|\, f^*\omega\\
	&= \int_\T \tilde \phi \circ \tilde f\, d\sigma - \log\|z\|\\
	&= H_{\tilde \phi}(\tilde f) - \log\|z\|.
\end{align*}
That is
\begin{equation}\label{eq:H}
	H_{\omega,\phi}(f) = H_{\tilde \phi}(\tilde f) - \log\|z\|.
\end{equation}

Now note that every disc $\tilde f \in \A_{\C^{n+1} \setminus \{0\}}^{\tilde W}$
gives a disc $f = \pi \circ \tilde f \in \A_{\P^{n}}^W$, and conversely for every disc
$f = \pi \circ \tilde f \in \A_{\P^{n}}^W$ there is a disc
$\tilde f \in \A_{\C^{n+1} \setminus \{0\}}^{\tilde W}$
such that $f = \pi \circ \tilde f$.

Hence, by taking the infimum over $f$ on the left hand side of (\ref{eq:H}) corresponds
to taking infimum over $\tilde f$ on the right hand side. This shows that 
\begin{equation}\label{eq:EH}
	E_W H_{\omega,\phi}(\pi(z)) = E_{\tilde W} H_{\tilde \phi}(z) + \log\|z\|.
\end{equation}

By Lemma \ref{lem:n+1_disc_structure},
$\tilde W$ admits a  $\tilde W$-disc
structure $\beta$ such that $E_{\mathcal B} H_{\tilde \phi} \leq \tilde \phi$,
and by Corollary \ref{cor:main} 
we have
\begin{equation}\label{eq:tilde_phi2}
  \sup \F_{\tilde \phi} = E_{\tilde W} H_{\tilde \phi}
\end{equation}
in $\C^{n+1}\setminus \{0\}$.

Finally, by (\ref{eq:EH}), (\ref{eq:tilde_phi2}), and Lemma \ref{eq:supFloghom},
we can show that $E_W H_{\omega,\phi} = \sup \F_{\omega,\phi}$:
$$
 E_W H_{\omega,\phi}(\pi(z)) 
 = E_{\tilde W} H_{\tilde \phi}(z) + \log\|z\| 
 = \sup \F_{\tilde \phi}(z) + \log\|z\| 
 = \sup \F_{\omega,\phi}(\pi(z)).
$$
\end{proof}

%%%%%%%%%%%%%%%%%%%%%%%%%%%%%%%%%%%%%%%%%%%%%%%%%%%%%%%%%%%%%%%%%%%%%%%%%%%%%%
%%%%%%%%%%%%%%%%%%%%%%%%%%%%%%%%%%%%%%%%%%%%%%%%%%%%%%%%%%%%%%%%%%%%%%%%%%%%%%
\section{Applications to projective hulls}\label{sec:projective_hulls}

The case when $\phi = 0$ in Theorem \ref{th:omega} 
is interesting in its own way since
it gives the global extremal function for the set $W$, 
\cite[\S 5 and 6]{GueZer:2005}.

\begin{definition}
 Let $E$ be a Borel subset in $\P^n$. The \emph{global extremal function}
 for $E$ is defined as
 $$
	\Lambda_E(x) = \sup\{ u(x) ; u \in \PSH(\P^n,\omega), u|_E \leq 0\}.
 $$
\end{definition}
In this case the when $E$ is a domain Theorem \ref{th:omega} gives
the following formula
$$
	\Lambda_E(x) = \inf\{ -\frac{1}{2\pi}\int_\D \log|\cdot|\, f^*\omega ; 
	f\in \A_{\P^n}^E, f(0) = x \}.
$$

\begin{definition}
 Let $K$ be a compact subset of $\P^n$. The \emph{projective hull} of
 $K$, denoted $\hat K$ is defined as all the points $x\in \P^n$
 for which there exists a constant $C_x$ such that
 \begin{equation}\label{eq:def_hull}
	 \| P(x)\| \leq C_x^d \sup_K \| P\|, \quad\text{ for all }
	 P \in H^0(\P^n, \O(d)).
 \end{equation}
 
\end{definition}

Just as the polynomial hull can be characterized by the Siciak-Zahariuta
extremal function, the projective 
hull can be characterized using the global extremal function.
\begin{proposition}{\cite[\S 4]{HarLaw:2006}}\label{prop:hatK_Lambda}
If $K\subset \P^n$ is compact then
 $$
	\hat K = \{ x\in \P^n ; \Lambda_K(x) < +\infty \}.
 $$
 
 Furthermore, or each $x$ the value $\exp(\Lambda_K(x))$ is equal to 
 the infimum of all $C_x$ such that (\ref{eq:def_hull}) holds.
\end{proposition}

For a constant $\Lambda \geq 0$ we let 
$$
	\hat K(\Lambda) = \{x \in \P^n ; \Lambda_K(x) \leq \Lambda\}.
$$
The projective hull can then be written as a union of the 
sets $\hat K(\Lambda)$.

The disc formula proved in Section \ref{sec:omega} 
is only for domains in $\P^n$ but not
compact sets. This forces us to take a sequence of open neighbourhoods of $K$,
and the following proposition allows us to take a limit to obtain
$\Lambda_K$.
\begin{proposition}\label{prop:Lambda_lim}
Assume $K\subset \P^n$ is a compact set and $(U_j)_j$ is a decreasing
sequence of open subsets in $\P^n$ such that $\cap_j U_j = K$. 
Then 
$$ 
	\Lambda_K = \lim_{j\to \infty} \Lambda_{U_j}.
$$
\end{proposition}
\begin{proof}
Note first that since $U_j$ is a decreasing sequence then the 
sequence $\Lambda_{U_j}$ is increasing, in particular 
$\lim_{j\to \infty} \Lambda_{U_j}$ exists.
%(although it might be $+\infty$).

Since each function $\Lambda_{U_j}$ is in $\PSH(\P^n,\omega)$ 
(see \cite[Theorem 5.2 and Proposition 5.6]{GueZer:2005}) and is 
$0$ on $U_j \supset K$, then $\Lambda_K \geq \Lambda_{U_j}$ and therefore
$\Lambda_K \geq \lim_{j\to \infty} \Lambda_{U_j}$.

Let $\epsilon >0$. Since each function 
$u \in \PSH(\P^n,\omega)$, $u|_K \leq 0$ is upper semicontinuous
there is a neighbourhood $U$ of $K$ such that 
$u|_U \leq \epsilon$. Find $U_{j_0}$ such that $U_{j_0} \subset U$, then
for $x \in X$,
$$
	u(x)-\epsilon \leq \Lambda_{U_{j_0}}(x) \leq \lim_{j\to \infty} \Lambda_{U_j},
$$
which implies, by taking supremum over $u$ and letting $\epsilon \to 0$, 
that $\Lambda_K \leq \lim_{j\to \infty} \Lambda_{U_j}$.
\end{proof}

By combining the disc formula for the global extremal function with 
Proposition \ref{prop:hatK_Lambda} we can get a new characterization of 
the projective hull for connected sets. 
The characterization is quantitative, that is it uses $\hat K(\Lambda)$,
just as the characterization by existence of currents 
\cite[Theorem 11.1]{HarLaw:2006}.

\begin{theorem}\label{th:hulls}
Let $\Lambda >0$
 For a point $x$ in a connected compact subset 
 $K\subset \P^n$ the following is equivalent
 \begin{itemize}
  \item[(A)] $x \in \hat K(\Lambda)$
  \item[(B)] For every $\epsilon >0$ and every neighbourhood $U$ of $K$ 
  there exists a disc
  $f \in \A_{\P^n}^U$ such that $f(0)=x$ and 
  $$
		- \frac{1}{2\pi}\int_\D \log|\cdot| f^*\omega < \Lambda + \epsilon.
	$$
 \end{itemize}
\end{theorem}

\begin{proof}
 First assume $x\in \hat K(\Lambda)$. By Proposition \ref{prop:Lambda_lim}
 there is a domain $V$ such that 
 $K \subset V \subset U$ and $\Lambda_V(x) < \Lambda_K(x) + \epsilon/2$.
 By Theorem \ref{th:omega} there is a disc $f\in \A_{\P^n}^{V}$
 such that $f(0) = x$ and 
 $$
	- \frac{1}{2\pi}\int_\D \log|\cdot|\, f^*\omega < \Lambda_V(x)  + \frac \epsilon 2 
 $$
 Then 
 $$
	- \frac{1}{2\pi}\int_\D \log|\cdot|\, f^*\omega  \leq \Lambda_V(x) + 
	\frac \epsilon 2 \leq \Lambda_K(x) + \epsilon \leq \Lambda + \epsilon.
 $$
 
 Conversely, assume \textsl{(B)} holds. Now let
 $U_j$ decreasing sequence of domains such that $\cap_j U_j = K$ and
 $\lim_{j \to \infty} \Lambda_{U_j} = \Lambda_K$. The sets $U_j$ can be chosen
 connected because $K$ is always contained in one connected component of 
 an open neighbourhood of $K$ (otherwise $K$ would not be connected).
 
 For each $j$ there is a disc $f_j \in \A_{\P^n}^{U_j}$ such that
 $f_j(0) = x$ and 
 $$
	\Lambda_{U_j}(x) \leq -\frac{1}{2\pi} \int_\D \log|\cdot| f_j^*\omega < \Lambda + \frac 1j
 $$
 which implies $\Lambda_K(x) = \lim_{j \to \infty} \Lambda_{U_j}(x) \leq \Lambda$, that is $x\in \hat K(\Lambda)$.
\end{proof}

%%%%%%%%%%%%%%%%%%%%%%%%%%%%%%%%%%%%%%%%%%%%%%%%%%%%%%%%%%%%%%%%%%%%%%%%%%%
\subsection{Characterizations of Drnov\v sek and Forstneri\v c}
In \cite{DrnFor:2012} Drnov\v sek and Forstneri\v c gave several characterizations
of the projective and polynomial hulls using the existence of analytic disc, both
for connected sets and not connected sets. 
They used  analytic disc in $\P^n$ with the bounded lifting property obtained from
Poletsky's disc formula and 
disc in $\C^{n+1}\setminus \{0\}$ derived from the 
disc formula of the Siciak-Zahariuta extremal function \cite{LarSig:2005}.

The result from \cite{DrnFor:2012} regarding connected sets which is best compatible with 
Theorem \ref{th:hulls} is the following.
\begin{theorem}[\cite{DrnFor:2012}, Theorem 5.1]\label{th:hullsBF}
 Let $K$ be a compact connected set in $\P^n$. A point $p\in \C^{n+1}\setminus \{0\}$
 belongs to the polynomial hull of the set $S_K \subset \C^{n+1}$, and hence
 $x = \pi(p) \in \P^n$ belongs to the projective hull of $K$, if and only if
 there exists a sequence of analytic discs 
 $F_j:\overline \D \to \C^{n+1}\setminus \{0\}$
 such that $F_j(0) = p$ and 
 $$
	\lim_{j \to \infty} \max_{t \in [0,2\pi]} dist(F_j(e^{it}, S_K) = 0.
 $$ 
\end{theorem}

The set $S_K$ is the lifting of $K$ restricted to the unit sphere in $\C^{n+1}$,
$S_K = \pi^{-1}(K) \cap \{ z \in \C^{n+1} ; \|z\| = 1 \}$.

\smallskip

Restating the theorem above without the limit and focusing on the point 
$x \in \P^n$  %in the same way as Theorem \ref{th:hulls} 
we can state it as follows.
\begin{theorem}\label{th:hulls2}
 For a point $x$ in a connected compact subset 
 $K\subset \P^n$ the following is equivalent.
 \begin{itemize}
  \item[(A')] $x \in \hat K$.
  \item[(B')] There exists $p \in \pi^{-1}(x)$ such that
		for every $\epsilon > 0$ and every neighbourhood $\tilde U$ of 
		$S_K$ there exists a disc $\tilde f \in \A_{\C^{n+1}\setminus \{0\}}^{\tilde U}$ 
		such that $\tilde f(0) = p$.
 \end{itemize}
\end{theorem}

Note that $x$ is in $\hat K$ if and only if there exists a point $p \in \pi^{-1}(x)$
which is in the polynomial hull of $S_K$ in $\C^{n+1}$ by
\cite[Corollary 5.3]{HarLaw:2006}.
Furthermore, if we define
$$
	\rho_K(x) = 
	\sup\{ \|p\| ; p \in \pi^{-1}(x) \text{ and $p$ is in the polynomial hull of } S_K \},
$$
then we have following connection between 
the global extremal function $\Lambda_K$, the best constant function $C_K$ and 
the function $\rho_K$, see \cite[Proposition 5.2]{HarLaw:2006},
\begin{equation}\label{eq:LambdaCrho}
	\Lambda_K = \log C_K = - \log \rho_K.
\end{equation}
The definition of $\rho_K$ implies that condition (B') holds for $p$ such that 
$\|p\| \leq \rho_K(x)$.

Since both condition (B) and (B') characterize the projective hull
they are clearly equivalent. But for completeness we show how 
one can be derived from the other.

\medskip

%%%%%%%%%%%%%%%%%%%%%%%%%%%%%%%%%%%%%%%%%%%%%%%%%%%%%%%%%%%%%%%%%%%%%%%%%%%%%%
\textbf{(B') implies (B):}
Let $x \in \hat K(\Lambda)$, where $\Lambda > \Lambda_K(x)$.
Fix $\epsilon >0$ and let $U$ be an open neighbourhood 
of $K$ in $\P^n$. From the definition of $\rho_K$ and (\ref{eq:LambdaCrho}) 
there is $p \in \pi^{-1}(x)$
such that
$-\log \|p\|  < \Lambda$. 
Then by (B') there is a disc $\tilde f$ such that $\tilde f(0) = p$ and
$$
	\max_{t\in \T}\,  \log\|\tilde f(t)\| \leq \epsilon.
$$	
That is, we let $\tilde U = \{ z \in \C^{n+1} ; \|z\| < \exp(\epsilon) \}$.
Now define $f = \pi \circ \tilde f$. 
Using the Riesz representation theorem for the function $\log\|\tilde f\|$
at the point 0 we get
\begin{align*}
	\log\|p\| &= \frac{1}{2\pi}\int_\D \log|\cdot|\, \Delta \log\|\tilde f\| + 
	\int_\T \log\|\tilde f\|\, d\sigma 
\end{align*}
Using the inequality $-\log \|p\| < \Lambda$ and the fact that 
$f^*\omega = \Delta \log\|\tilde f\|$ we derive the following.
\begin{align*}
	-\frac{1}{2\pi}\int_\D \log|\cdot|\, f^*\omega 
	&= -\frac{1}{2\pi}\int_\D \log|\cdot|\, \Delta \log\|\tilde f\| \\
	&= -\log\|p\| + \int_\T \log\|\tilde f\|\, d\sigma \\
	&\leq \Lambda + \max_{t \in \T}\, \log\|\tilde f(t)\| \\
	&\leq \Lambda + \epsilon.
\end{align*}

%%%%%%%%%%%%%%%%%%%%%%%%%%%%%%%%%%%%%%%%%%%%%%%%%%%%%%%%%%%%%%%%%%%%%%%%%%%%%%
\textbf{(B) implies (B'):}
Given a neighbourhood $\tilde U$ of $S_K$. 
Let $U = \pi(\tilde U)$ and $f\in \A_{\P^n}^U$ be a disc as in (B).
Then there is a lifting 
$\tilde f_0 \in \A_{\C^{n+1}\setminus \{0\}}$
such that $f = \pi \circ \tilde f_0$, see e.g.~\cite[\S 3]{DrnFor:2012}. 
Note that $\tilde f_0(\T) \subset \pi^{-1}(U)$ since $f(\T) \subset U$.

Now let $u$ be the solution of the
Dirichlet problem on the unit disc with boundary values $\log\|\tilde f_0\|$,
and let $v$ be the corresponding harmonic conjugate. Note that $v$ is harmonic
on $\D$ but not necessarily continuous up to the boundary. 
However, since 
$$
	\left\| \frac{\tilde f_0(t)}{\exp(u(t))} \right\| = 
	\left\| \frac{\tilde f_0(t)}{\tilde f(t)}\right\| = 1, 
	\qquad \text{for all } t \in \T,
$$
there is $r < 1$ such that 
$\tilde f_0(t) /\exp(u(rt)) \in \tilde U$ for all $t\in \T$.
If we then define the closed analytic disc $\tilde f$ by 
$$
	\tilde f(t) = \frac{\tilde f_0(t)}{\exp(u(rt) + iv(rt))}
$$
we see that $\tilde f$ maps the unit circle into $\tilde U$ and $p = \tilde f(0)
\in \pi^{-1}(x)$. Since $r<1$ the function $\exp(u(r\cdot) + iv(r\cdot))$ is 
continuous on $\overline \D$ and therefore 
$\tilde f \in \A_{\C^{n+1}\setminus\{0\}}^{\tilde U}$.

Furthermore, if $x \in \hat K(\Lambda)$, $\Lambda > 0$ and $\epsilon >0$, 
then we can select $f$ such that 
$$
		- \frac{1}{2\pi}\int_\D \log|\cdot| f^*\omega < \Lambda + \epsilon,
$$
and we can select $r$ close enough to 1 so that 
$$
	-\epsilon < \int_\T \log\|\tilde f(t)\|\, d\sigma < \epsilon.
$$

This implies, by similar calculations as before, that
\begin{align}\label{ineq:rho_eps}
	-\log\|p\| 
	&= -\frac{1}{2\pi}\int_\D \log|\cdot|\, \Delta \log\|\tilde f\| 
		- \int_\T \log \|\tilde f\|\, d\sigma \nonumber\\ 
	&\leq - \frac{1}{2\pi}\int_\D \log|\cdot|\, f^*\omega + \epsilon \nonumber\\
	&\leq \Lambda + 2 \epsilon \nonumber\\
	&\leq \Lambda_K(x) + 2\epsilon = -\log \rho_K(x) + 2\epsilon.
\end{align}
By definition we have $\|p\| \leq \rho_K(x)$. This, along with 
inequality (\ref{ineq:rho_eps}) above, shows that
\begin{equation}\label{ineq:p_rho}
	e^{2\epsilon} \rho_K(x) \leq \|p\| \leq \rho_K(x).
\end{equation}
In other words, we can choose $p$ such that $\|p\|$ is arbitrary close 
to $\rho_K(x)$.

\newpage

\textbf{Remark:} 
There are two details that should be pointed out here:
\begin{itemize}
 \item The results in \cite{DrnFor:2012} and Theorem \ref{th:hulls2} are not 
quantitative, in other words they do not use $\Lambda$ and $\hat K(\Lambda)$
as Theorem \ref{th:hulls} does. 
 \item Theorem \ref{th:hullsBF} implies that for every $p$ in the 
polynomial hull of $S_K$ there exists a disc with the properties in (B') and
center $p$. On the other hand, Theorem \ref{th:hulls} does not allow us to select the 
center $p$ specifically. It only allows us to find
discs such that the modulus of its center is as close to $\rho_K(x)$ as we wish,
as seen from (\ref{ineq:p_rho}).
\end{itemize}

\bigskip

The author wishes to thank Franc Forstneri\v c and Josip Globevnik for referring him
to the results in 
\cite{DrnFor:2012} and asking about the connections between the characterizations
of the projective hull there and the one in Theorem \ref{th:hulls}.

%%%%%%%%%%%%%%%%%%%%%%%%%%%%%%%%%%%%%%%%%%%%%%%%%%%%%%%%%%%%%%%%%%%%%%%%%%%%%%
%%%%%%%%%%%%%%%%%%%%%%%%%%%%%%%%%%%%%%%%%%%%%%%%%%%%%%%%%%%%%%%%%%%%%%%%%%%%%%
\section{Siciak-Zahariuta extremal function and other global extremal functions}\label{sec:S-Z}
There are other quasiplurisubharmonic function of interest in $\P^n$. The most
known are those when the current is the current of integration $[H_\infty]$ for the 
hyperplane at infinity $H_\infty$. Then we are looking at $\P^n$ as the
union of $\C^n$ and $H_\infty$. The class of quasiplurisubharmonic functions
$\PSH(\P^n,[H_\infty])$ then becomes the Lelong-class 
$$
	\L = \{ u \in \PSH(\C^n); u(x) \leq \log\|x\| + C_u \}.
$$

The potential for the pullback of this current to $\C^{n+1} \setminus \{0\}$, 
denoted $\pi^* [H_\infty]$,
then has a global potential and can be written as 
$\pi^* [H_\infty] = dd^c \log|z_0|$, assuming 
$H_\infty = \pi(\{ z \in \C^{n+1}\setminus \{0\} ; z_0 = 0 \})$.

The \emph{Siciak-Zahariuta extremal function} for a set $W$ is defined 
as 
$$ 
	\sup \{ u(x) ; u \in \L, u|_W \leq 0 \},
$$
and the weighted version as 
$$ 
	\sup \{ u(x) ; u\in \L, u|W \leq \phi \},
$$
where $\phi:W \to \overline \R$ is a function.

If $W\subset \C^n$ is a domain and $\phi:W\to \R\cup \{-\infty\}$ is an 
upper semicontinuous function then there is a disc formula for the 
Siciak-Zahariuta function \cite{LarSig:2005, MagSig:2007},
\begin{multline*}
%  \sup\{ u(x) ; u \in \L, u|_W \leq \phi \} = \\
  \sup\L_\phi (x) = \\
  \inf \left\{ -\sum_{a \in f^{-1}(H_\infty) } \log|a| +
  \int_\T \phi \circ f\, d\sigma\ ;\ f \in \A_{\P^n}^W, f(0) = x \right\}.
\end{multline*}
There is also a formula when $W$ is not connected \cite{LarSig:2009}, 
but it is a little bit different and somewhat unwieldier. 

The formula above can be proven easily by the same methods as the formula in
Theorem \ref{th:omega} by replacing $\log\|z\|$ with $\log|z_0|$.

We only have to note two things. First, a function $u \in \L$ extends to a 
plurisubharmonic and logarithmically homogeneous function 
$\tilde u: \C^{n+1}\setminus \{0\} \to \R \cup \{-\infty\}$, by
\begin{equation}\label{eq:L_lifting}
  \tilde u(z_0,z_1,\ldots,z_n) 
  = u \left(\frac{(z_1,\ldots,z_n)}{z_0}\right) + \log|z_0| 
\end{equation}

Secondly, 
if $f = [f_0:f_1:\cdots:f_n] \in \A_{\P^n}$, then 
$$\int_D \log|\cdot|\, f^*[H_\infty] 
= \sum_{a \in f^{-1}(H_\infty)} m_a \log|a|,
$$
where $m_a$ is the multiplicity of the zero of $f_0$ at $a$.
However, by Proposition 1 in \cite{LarSig:2005} the multiplicity $m_a$ 
can by omitted, because for a disc $f$ with zero of 
order $m_a$ at $a$ there is disc with $m_a$ different simple zeros 
sufficiently close to $a$. This implies that the multiplicity $m_a$ 
can be omitted in the disc formula above.

\medskip

\textbf{Remark:}
The methods described here actually apply to every current $\tilde \omega$ 
on $\P^n$, such that the pullback to $\C^{n+1}\setminus \{0\}$, 
$\pi^*\tilde \omega$, 
has a logarithmically
homogeneous potential 
$\psi:\C^{n+1}\setminus \{0\} \to \R \cup \{-\infty\}$ 
with $dd^c \psi = \pi^*\tilde \omega$.

%%%%%%%%%%%%%%%%%%%%%%%%%%%%%%%%%%%%%%%%%%%%%%%%%%%%%%%%%%%%%%%%%%%%%%%%%%%%%%
%%%%%%%%%%%%%%%%%%%%%%%%%%%%%%%%%%%%%%%%%%%%%%%%%%%%%%%%%%%%%%%%%%%%%%%%%%%%%%
%\listoftodos

\bibliographystyle{siam}
\bibliography{bibref}

\end{document}